\documentclass[USenglish,10pt]{article}

\usepackage[utf8]{inputenc}
\usepackage[small]{dgruyter}
\usepackage{amsmath,amsfonts,amsthm,amssymb,amsxtra,bbm,fourier,graphicx,color,setspace}

\bibpunct{[}{]}{,}{}{}{,}
\setcitestyle{numbers}

\newtheorem{theorem}{Theorem}
\newtheorem{proposition}[theorem]{Proposition}
\newtheorem{lemma}[theorem]{Lemma}
\newtheorem{corollary}[theorem]{Corollary}\newtheorem{remark}[theorem]{Remark}

\newcommand{\R}{{\mathbb R}}
\renewcommand{\S}{{\mathbb S}}

\newcommand{\be}[1]{\begin{equation}\label{#1}}
\newcommand{\ee}{\end{equation}}
\renewcommand{\(}{\left(}
\renewcommand{\)}{\right)}
\newcommand{\iSd}[1]{\int_{\S^d}{#1}\,d\mu}
\newcommand{\nrmSd}[2]{\|{#1}\|_{\mathrm L^{#2}(\S^d)}}
\newcommand{\bignrmSd}[2]{\big\|{#1}\big\|_{\mathrm L^{#2}(\S^d)}}
\newcommand{\iRd}[1]{\int_{\R^d}{#1}\,dx}
\newcommand{\cb}[1]{\langle{#1}\rangle}

\newcommand{\nc}{\normalcolor}

\begin{document}
\journalname{} 
\author*[1]{Jean Dolbeault}
\author[2]{Maria J.~Esteban}
\runningauthor{J.~Dolbeault \& M.J.~Esteban}
\affil[1]{CEREMADE (CNRS UMR n$^\circ$ 7534),\newline PSL university, Universit\'e Paris-Dauphine, Place de Lattre de Tassigny, F-75775 Paris 16}
\affil[2]{CEREMADE (CNRS UMR n$^\circ$ 7534),\newline PSL university, Universit\'e Paris-Dauphine, Place de Lattre de Tassigny, F-75775 Paris 16}
\title{Improved interpolation inequalities and stability}
\runningtitle{Improved interpolation inequalities and stability}
\abstract{For exponents in the subcritical range, we revisit some optimal interpolation inequalities on the sphere with \emph{carré du champ} methods and use the remainder terms to produce improved inequalities. The method provides us with lower estimates of the optimal constants in the symmetry breaking range and stability estimates for the optimal functions. Some of these results can be reformulated in the Euclidean space using the stereographic projection.}
\dedication{L.~V\'eron on the occasion of his 70$^{th}$ anniversary.}
\maketitle

\section{Introduction}\label{Sec:Intro}

Let us consider the sphere $\S^d$ endowed with the uniform probability measure $d\mu$. We shall define by $\nrmSd uq=\(\kern4pt\iSd{|u|^q}\)^{1/q}$ the corresponding norm, denote by $2^*$ the critical exponent in dimension $d\ge3$, that is, $2^*=2\,d/(d-2)$ and adopt the convention that $2^*=\infty$ if $d=1$ or $d=2$. The subcritical \emph{Gagliardo-Nirenberg inequalities} on the sphere of dimension $d$ can be stated as follows: for $p\in(2,2^*)$,
\be{Ineq:GNSgen1}
\frac{p-2}d\,\nrmSd{\nabla u}2^2+\lambda\,\nrmSd u2^2\ge\mu(\lambda)\nrmSd up^2\quad\forall\,u\in\mathrm H^1(\S^d,d\mu)\,,
\ee
where the function $\lambda\mapsto\mu(\lambda)$ is positive, concave, increasing and such that $\mu(\lambda)=\lambda$ for $\lambda\in(0,1]$ and $\mu(\lambda)<\lambda$ if $\lambda>1$: see~\cite{DolEsLa-APDE2014}. Moreover, if $\lambda\in(0,1]$, the only extremals of~\eqref{Ineq:GNSgen1} are the constant functions. In the limit case $p=2^*$, with \hbox{$d\ge3$}, the inequality also holds with optimal constant $\mu(\lambda)=\min\{\lambda,1\}$ and it is simply the \emph{Sobolev inequality} on $\S^d$ when $\lambda=1$.

In the case $p\in[1,2)$, as shown in~\cite{DolEsLa-APDE2014}, there are similar inequalities where the roles of $p$ and $2$ are exchanged: for $p\in[1,2)$,
\be{Ineq:GNSgen2}
\frac{2-p}d\,\nrmSd{\nabla u}2^2+\mu\nrmSd up^2\ge\lambda(\mu)\,\nrmSd u2^2\quad\forall\,u\in\mathrm H^1(\S^d,d\mu)\,.
\ee
Here the function $\mu\mapsto\lambda(\mu)$ is positive, concave, increasing and such that
$\lambda(\mu)=\mu$ for $\mu\in(0,1]$, and $\lambda(\mu)<\mu$ if $\mu>1$. If $\mu\in(0,1]$, the only extremals of~\eqref{Ineq:GNSgen2} are the constant functions. In the limit case $p=1$, the inequality with $\lambda=1$ is the \emph{Poincar\'e inequality}.

With $\lambda=1$, Inequalities~\eqref{Ineq:GNSgen1} and~\eqref{Ineq:GNSgen2} can be rewritten as
\be{Ineq:GNS}
\nrmSd{\nabla u}2^2\ge\frac d{p-2}\(\nrmSd up^2-\nrmSd u2^2\)\quad\forall\,u\in\mathrm H^1(\S^d,d\mu)
\ee
for any $p\in[1,2)\cup(2,2^*)$ if $d=1$, $2$, and for any $p\in[1,2)\cup(2,2^*]$ if $d\ge3$. Since $d\mu$ is a probability measure, we know from H\" older's inequality that the right-hand side of~\eqref{Ineq:GNS} is nonnegative independently of the sign of $(p-2)$. We will call~\eqref{Ineq:GNS} the \emph{Gagliardo-Nirenberg-Sobolev interpolation inequality}. In the case $p>2$, it is usually attributed to W.~Beckner~\cite{MR1230930} but can also be found in~\cite[Corollary~6.1]{BV-V}. However an earlier version corresponding to the range $p\in[1,2)\cap(2,2^\#)$ was established in the context of continuous Markov processes and linear diffusion operators by D.~Bakry and M.~Emery in~\cite{Bakry-Emery85,MR808640}, using the \emph{carr\'e du champ} method, where~$2^\#$ is the \emph{Bakry-Emery exponent} defined as
\[
2^\#=\frac{2\,d^2+1}{(d-1)^2}
\]
for any $d\ge2$, and where we shall adopt the convention that $2^\#=+\infty$ if $d=1$. Notice that the case $p=2^\#$ is also covered in~\cite{MR808640,Bakry-Emery85} if $d\ge2$. By taking the limit in~\eqref{Ineq:GNS} as $p\to2$, we obtain the \emph{logarithmic Sobolev inequality} on $\S^d$,
\be{Ineq:logSob}
\nrmSd{\nabla u}2^2\ge\frac d2\,\iSd{|u|^2\,\log\(\frac{|u|^2}{\nrmSd u2^2}\)}\quad\forall\,u\in\mathrm H^1(\S^d,d\mu)\setminus\{0\}\,.
\ee
For brevity, we shall consider it as the ``$p=2$ case'' of the Gagliardo-Nirenberg-Sobolev interpolation inequality. Inequality~\eqref{Ineq:logSob} was known from earlier works, see for instance~\cite{MR674060}.

Various proofs of~\eqref{Ineq:GNS} have been published. By Schwarz foliated symmetrization, it is possible to reduce~\eqref{Ineq:GNS} to inequalities based on the ultraspherical operator, which simplifies a lot the computations: see~\cite{DEKL2012,DEKL,1504} and references therein for earlier results on the ultraspherical operator. In this paper, we rely on the \emph{carr\'e du champ} method of D.~Bakry and M.~Emery and refer to~\cite{MR3155209} for a general overview of this technique. We also revisit some \emph{improved Gagliardo-Nirenberg-Sobolev inequalities} that can be written~as
\be{improved}
\nrmSd{\nabla u}2^2\ge d\,\varphi\(\frac{\nrmSd up^2-\nrmSd u2^2}{(p-2)\,\nrmSd up^2}\)\,\nrmSd up^2\quad\forall\,u\in\mathrm H^1(\S^d)\,.
\ee
Here $\varphi$ is a nonnegative convex function such that $\varphi(0)=0$ and $\varphi'(0)=1$. As a consequence, $\varphi(s)\ge s$ and we recover~\eqref{Ineq:GNS} if $\varphi(s)\equiv s$, but in improved inequalities we will have $\varphi(s)>s$ for all $s\neq0$. Such improvements have been obtained in~\cite{MR2381156,DEKL,Dolbeault20141338,1504}. Here we write down more precise estimates and draw some interesting consequences of~\eqref{improved}, such as lower estimates for the best constants in~\eqref{Ineq:GNSgen1} and~\eqref{Ineq:GNSgen2} or improved weighted Gagliardo-Nirenberg inequalities in the Euclidean space $\R^d$.

The improved inequality~\eqref{improved}, with $\varphi(s)>s$ for $s\neq0$, can also be considered as a stability result for~\eqref{Ineq:GNS} in the sense that it can also be rewritten as
\[
\nrmSd{\nabla u}2^2-\frac d{p-2}\(\nrmSd up^2-\nrmSd u2^2\)\ge d\,\psi\(\frac{\nrmSd up^2-\nrmSd u2^2}{(p-2)\,\nrmSd up^2}
\)\,\nrmSd up^2
\]
for any $u\in\mathrm H^1(\S^d)$, with $\psi(s)=\varphi(s)-s>0$ for $s\neq0$. Here the right-hand side of the inequality is a measure of the distance to the optimal functions, which are the constant functions: see Appendix~\ref{App:Distance} for details.

\section{Main results}\label{Sec:Main}

Our first result goes as follows. Let
\be{gamma1}
\gamma=\(\frac{d-1}{d+2}\)^2\,(p-1)\,(2^\#-p)\quad\mbox{if}\quad d\ge2\,,\quad\gamma=\frac{p-1}3\quad\mbox{if}\quad d=1\,,
\ee
so that $\gamma=2-p$ with $1\le p\le2^\#$ means that
\[\label{gammaLim}
\begin{array}{ll}
d=1\quad\mbox{and}\quad p=7/4=p_*(1)\,,\\
d>1\quad\mbox{and}\quad p=p_*(d)\,
\end{array}
\]
occurs, where
\[
p_*(d)=\frac{3+d+2\,d^2-2\,\sqrt{4\,d+4\,d^2+d^3}}{(d-1)^2}
\]
for any $d\ge2$. Notice that for all $d\ge1$, $1<p_*(d)<2$ and $\lim_{d\to+\infty}p_*(d)=2$. For any admissible $s\ge0$, \emph{i.e.}, for any $s\in\big[0,(p-2)^{-1}\big)$ if $p>2$ and any $s\ge0$ if $p\in[1,2)$, let
\be{phifunction}
\begin{array}{ll}
\varphi(s)=\frac{1-(p-2)\,s-\(1-(p-2)\,s\)^{-\frac\gamma{p-2}}}{2-p-\gamma}\quad&\mbox{if}\quad\gamma\neq2-p\,,\\[4pt]
\varphi(s)=\frac1{2-p}\(1+(2-p)\,s\)\log\(1+(2-p)\,s\)\quad&\mbox{if}\quad\gamma=2-p\,.
\end{array}
\ee
Written in terms of $\nrmSd u2^2$ and $\nrmSd up^2$, we shall prove in Section~\ref{Sec:Heat} that~\eqref{improved} holds with $\varphi$ given by~\eqref{phifunction} and gives rise to a following, new interpolation inequality.
\begin{theorem}\label{Thm:Interp} Let $d\ge1$, assume that
\be{Range:p}
p\neq2\,,\quad\mbox{and}\quad1\le p\le2^\#\quad\mbox{if}\quad d\ge2\,,\quad p\ge1\quad\mbox{if}\quad d=1
\ee
and let $\gamma$ be given by~\eqref{gamma1}. Then we have
\be{improvedineq}
\nrmSd{\nabla u}2^2\ge\frac d{2-p-\gamma}\(\nrmSd u2^2-\nrmSd up^{2-\frac{2\,\gamma}{2-p}}\,\nrmSd u2^{\frac{2\,\gamma}{2-p}}\)\quad\forall\,u\in\mathrm H^1(\S^d)
\ee
if $\gamma\neq2-p$, and
\be{improvedineqLog}
\nrmSd{\nabla u}2^2\ge\frac{2\,d}{p-2}\,\nrmSd u2^2\,\log\(\frac{\nrmSd u2^2}{\nrmSd up^2}\)\quad\forall\,u\in\mathrm H^1(\S^d)
\ee
if $\gamma=2-p$.\end{theorem}
In Inequalities~\eqref{improvedineq} and~\eqref{improvedineqLog}, the equality case is achieved by constant functions only and the constants $\frac d{2-p-\gamma}$ in~\eqref{improvedineq} and $\frac{2\,d}{p-2}$ in~\eqref{improvedineqLog} are sharp as can be shown by testing the inequality with $u=1+\varepsilon\,v$ with $v$ such that $-\Delta v=d\,v$ in the limit as $\varepsilon\to0$.

\medskip Now, let us come back to~\eqref{Ineq:GNSgen1} and~\eqref{Ineq:GNSgen2}. We deduce from Theorem~\ref{Thm:Interp} the following estimates of the best constants in \eqref{Ineq:GNSgen1} and \eqref{Ineq:GNSgen2}: see Fig.~\ref{F1} for an illustration.
\begin{theorem}\label{Thm:Lambda-Mu} Let $d\ge1$, $\gamma$ be given by~\eqref{gamma1} and assume that $p$ is in the range~\eqref{Range:p}.\\[4pt]
\emph{(i)} If $1\le p<2$, $p\neq p_*(d)$, then
\[
\lambda(\mu)\ge\frac{2-p-\gamma\,\mu^{1-\frac{2-p}\gamma}}{2-p-\gamma}\quad\forall\,\mu\ge1\,.
\]
\emph{(ii)} If $2<p<2^\#$, then
\[
\mu(\lambda)\ge\(\lambda+\frac{p-2}\gamma\,(\lambda-1)\)^\frac\gamma{\gamma+p-2}\quad\forall\,\lambda\ge1\,.
\]
\end{theorem}

Our third result has to do with stability for inequalities in the Euclidean space~$\R^d$ with $d\ge2$. For all $x\in\R^d$, let us define $\cb x:=\sqrt{1+|x|^2}$ and recall that $\big|\S^d\big|=2\,\pi^\frac{d+1}2/\Gamma\(\frac{d+1}2\)$. Using the stereographic projection of $\S^d$ onto $\R^d$ (see Appendix~\ref{App:Stereographic}), Inequality~\eqref{Ineq:GNS} can be written as a weighted interpolation inequality in $\R^d$:
\[
\iRd{|\nabla v|^2}+\frac{d\,\delta(p)}{p-2}\iRd{\frac{|v|^2}{\cb x^4}}\ge\mathcal C_{d,p}\(\iRd{\frac{|v|^p}{\cb x^{\delta(p)}}}\)^\frac2p\;\mbox{with}\;\mathcal C_{d,p}=\frac{2^\frac{\delta(p)}p\,d\,\big|\S^d\big|^{1-\frac2p}}{p-2}
\]
where
\[
\delta(p)=2\,d-p\,(d-2)\,.
\]
Notice that $\delta(2^*)=0$ for any $d\ge3$, so that the inequality is the Sobolev inequality with sharp constant if $p=2^*$. However, for any $p\in[1,2)\cup(2,2^*]$ and \hbox{$d\ge3$}, equality is obtained with $v_\star(x)=\cb x^{2-d}$ and this function is, up to an arbitrary multiplicative constant, the only one to realize the equality case if $p<2^*$. Equality is achieved by $v_\star=1$ in dimension $d=2$ for any $p\in[1,2)\cup(2,+\infty)$. Let us notice that $\nabla v_\star$ is not in $\mathrm L^2(\R^d)$ if $d=1$. Using the improved version~\eqref{improvedineq} of the inequality, we obtain as in Theorem~\ref{Thm:Interp} the following stability result.
\begin{theorem}\label{Thm:Stability} Let $d\ge2$ and assume that $p\in(2,2^\#)$. Then
\begin{multline*}
\iRd{|\nabla v|^2}+\frac{d\,\delta(p)}{p-2}\iRd{\frac{|v|^2}{\cb x^4}}-\mathcal C_{d,p}\(\iRd{\frac{|v|^p}{\cb x^{\delta(p)}}}\)^{2/p}\\
\ge\frac\gamma{p-2}\,\frac{\mathcal C_{d,p}}2\,\frac{\left[\(\iRd{\frac{|v|^p}{\cb x^{\delta(p)}}}\)^{2/p}-2^{2-\frac{\delta(p)}p}\,\big|\S^d\big|^{\frac2p-1}\iRd{\frac{|v|^2}{\cb x^4}}\right]^2}{\(\iRd{\frac{|v|^p}{\cb x^{\delta(p)}}}\)^{2/p}}
\end{multline*}
for any $v\in\mathrm L^2\big(\R^d,\cb x^{-4}\,dx\big)$ such that $\nabla v\in\mathrm L^2(\R^d,dx)$.\end{theorem}
Again, the right-hand side of the inequality is a measure of the distance to $v_\star$. The proof is elementary. With $\varphi$ given by~\eqref{phifunction} and $\psi(s)=\varphi(s)-s$, we notice that
\[
\psi''(s)\ge\gamma\,\big(1-(p-2)\,s\big)^{\frac\gamma{2-p}-2}
\]
for any admissible $s\ge0$. With $1=\nrmSd up^2\ge\nrmSd u2^2=1-(p-2)\,s$ and $\frac\gamma{2-p}-2<0$, we know that $\psi''(s)\ge\gamma$. As a consequence, we have
\[
\nrmSd{\nabla u}2^2-\frac d{p-2}\(\nrmSd up^2-\nrmSd u2^2\)\ge\frac{\gamma\,d}{2\,(p-2)^2}\,\frac{\(\nrmSd up^2-\nrmSd u2^2\)^2}{\nrmSd up^2}\,.
\]
The result of Theorem~\ref{Thm:Stability} follows by applying the stereographic projection. A sharper result valid also if $p\in[1,2)$ will be given in Proposition~\ref{Prop:Stability}.

\medskip As noticed in~\cite[Theorem~2.2]{1504}, in the Bakry-Emery range~\eqref{Range:p}, we obtain an improvement if we assume an orthogonality condition on the sphere. Let us recall the result, which is independent of what we have obtained so far. Let $\mathrm H^1_+(\S^d,d\mu)$ denote the set of the a.e.~nonnegative functions in $\mathrm H^1(\S^d,d\mu)$ and define
\[\label{LambdaStar}
\Lambda^\star(p)=\inf\frac{\nrmSd{\nabla u}2^2}{\nrmSd{u-1}2^2}
\]
where the infimum is taken on the set of the functions $u\in\mathrm H^1_+(\S^d,d\mu)$ such that $\iSd u=1$ and $\iSd{x\,|u|^p}=0$. Then for any $p\in(2,2^\#)$, we have
\[\label{Ineq:GNSimproved}
\nrmSd{\nabla u}2^2\ge\frac1{p-2}\(d+\frac{(d-1)^2}{d\,(d+2)}\(2^\#-p\)\(\Lambda^\star(p)-d\)\)\(\nrmSd up^2-\nrmSd u2^2\)
\]
for any function $u\in\mathrm H^1(\S^d,d\mu)$ such that $\iSd{x_i\,|u|^p}=0$ with $i=1,\,2,\ldots d$. We know from~\cite{1504} that $\Lambda^\star(p)>d$ but the value is not explicit except for the limit case $p=2$. In this case, the inequality becomes a logarithmic Sobolev inequality, which has been stated in~\cite[Proposition~5.4]{1504}. Using the stereographic projection, we obtain new inequalities on $\R^d$ which are as follows.
\begin{theorem}\label{Thm:AFST} Let $d\ge2$ and assume that $p\in(2,2^\#)$. Then
\begin{multline*}
\iRd{|\nabla v|^2}+\frac{d\,\delta(p)}{p-2}\iRd{\frac{|v|^2}{\cb x^4}}-\mathcal C_{d,p}\(\iRd{\frac{|v|^p}{\cb x^{\delta(p)}}}\)^{2/p}\\
\ge
\frac{(d-1)^2}{d\,(d+2)}\,\frac{2^\#-p}{p-2}\(\Lambda^\star(p)-d\)\,\left[2^\frac{\delta(p)}p\,\big|\S^d\big|^{1-\frac2p}\(\iRd{\frac{|v|^p}{\cb x^{\delta(p)}}}\)^{2/p}-\,4\iRd{\frac{|v|^2}{\cb x^4}}\right]
\end{multline*}
for any function $v$ in the space $\left\{v\in\mathrm L^2\big(\R^d,\cb x^{-4}\,dx\big)\,:\,\nabla v\in\mathrm L^2(\R^d,dx)\right\}$ such that
\[
\iRd{\frac{x}{\cb x^4}\,|v|^2}=0\quad\mbox{and}\quad\iRd{\frac{|x|^2}{\cb x^4}\,|v|^2}=\iRd{\frac{|x|^2}{\cb x^4}\,|v_\star|^2}\,.
\]
Under the same conditions on $v$, we also have
\[
\iRd{|\nabla v|^2}\ge d\,(d-2)\iRd{\frac{|v|^2}{\cb x^4}}+\,\frac\lambda2\iRd{\frac{|v|^2}{\cb x^4}\,\log\(\frac{\(\frac12\,\cb x^2\)^{d-2}\,|v|^2}{4\,\big|\S^d\big|^{-1}\iRd{\frac{|v|^2}{\cb x^4}}}\)}
\]
with $\displaystyle\lambda=d+\frac 2d\,\frac{4\,d-1}{2\,(d+3)+\sqrt{2\,(d+3)\,(2\,d+3)}}$.
\end{theorem}
Notice that the right-hand side of each of the two inequalities is proportional to the corresponding entropy and not to the square of the entropy as in Theorem~\ref{Thm:Stability}. This result is a counterpart for $p\in(2,2^\#)$, with a \emph{quantitative} constant, of the result of G.~Bianchi and H.~Egnell in~\cite{MR1124290} for the critical exponent $p=2^*$. See Remark~\ref{Rem:Bianchi-Egnell}. The constant $\Lambda^\star(p)$ can be estimated explicitly in the limit case as $p=2$: see~\cite[Proposition 5.4]{1504} for further details.

\medskip So far, all results have been limited to the Bakry-Emery range and rely on heat flow estimates on the sphere. However, using nonlinear flows as in~\cite{1504}, improvements and stability results can also be achieved when $p\in[2^\#,2^*)$. This will be the topic of Section~\ref{Sec:NonlinearFlows} while all results of Section~\ref{Sec:Main} are proved in Section~\ref{Sec:Heat} using the heat flow and the \emph{carr\'e du champ} method on the sphere.

\section{Heat flow and \emph{carr\'e du champ} method}\label{Sec:Heat}

In this section, our goal is to prove that~\eqref{improved} holds with $\varphi$ given by~\eqref{phifunction}.

\medskip In its simplest version, the \emph{carr\'e du champ} method goes as follows. We define the \emph{entropy} and the \emph{Fisher information} respectively by
\[
\mathsf e:=\frac 1{p-2}\(\nrmSd up^2-\nrmSd u2^2\)\quad\mbox{and}\quad\mathsf i:=\nrmSd{\nabla u}2^2\,.
\]
Then we shall assume that these quantities are driven by the flow such that $u^p$ is evolved by the heat equation, that is, we shall assume that $u>0$ solves
\be{HeatFlow}
\frac{\partial u}{\partial t}=\Delta u+(p-1)\,\frac{|\nabla u|^2}u
\ee
where $\Delta$ denotes the Laplace-Beltrami operator on $\S^d$. In the next result, $'$ denotes a~$t$~derivative.
\begin{lemma}\label{Lem:ODE} Let $d\ge1$, $\gamma$ be given by~\eqref{gamma1} and assume that $p$ is in the range~\eqref{Range:p}. With the above notations, $\mathsf e$ solves
\be{BEode}
\mathsf e''+2\,d\,\mathsf e'-\frac{\gamma\,|\mathsf e'|^2}{1-(p-2)\,\mathsf e}\ge0\,.
\ee\end{lemma}
\begin{proof} Since~\eqref{HeatFlow} amounts to $\frac{\partial u^p}{\partial t}=\Delta u^p$, it is straightforward to check that
\[
\frac d{dt}\iSd{|u(t,\cdot)|^p}=0\quad\mbox{and}\quad\mathsf e'=-\,2\,\mathsf i\,.
\]
Let us summarize results that can be found in~\cite{MR2381156,DEKL,Dolbeault20141338,1504}. We adopt the presentation of the proof of~\cite[Lemma~4.3]{Dolbeault2017133}. With $\S^d$ considered as a $d$-dimensional compact manifold with metric $g$ and measure $d\mu$, let us introduce some notation. If $\mathrm A_{ij}$ and $\mathrm B_{ij}$ are two tensors, then
\[
\mathrm A:\mathrm B:=g^{im}\,g^{jn}\,\mathrm A_{ij}\,\mathrm B_{mn}\quad\mbox{and}\quad\|\mathrm A\|^2:=\mathrm A:\mathrm A\,.
\]
Here $g^{ij}$ is the inverse of the metric tensor, \emph{i.e.}, $g^{ij}\,g_{jk}=\delta^i_{\!k}$. We use the Einstein summation convention and $\delta^i_{\!k}$ denotes the Kronecker symbol. Let us denote the \emph{Hessian} by $\mathrm H u$ and define the \emph{trace-free Hessian} by
\[
\mathrm L u:=\mathrm H u-\frac1d\,(\Delta u)\,g\,.
\]
We also define the trace-free tensor
\[
\mathrm M u:=\frac{\nabla u\otimes\nabla u}u-\frac1d\,\frac{|\nabla u|^2}u\,g\,.
\]
An elementary but lengthy computation that can be found in~\cite{Dolbeault2017133} shows that
\[\label{restgamma}
\frac12\(\mathsf i-d\,\mathsf e\)'=\frac12\(\mathsf i'+2\,d\,\mathsf i\)=-\,\frac d{d-1}\iSd{\left\|\mathrm L u-(p-1)\,\frac{d-1}{d+2}\,\mathrm M u\right\|^2}-\gamma\iSd{\frac{|\nabla u|^4}{u^2}}
\]
where $\gamma$ is given by~\eqref{gamma1}. In the framework of the \emph{carr\'e du champ} method of D.~Bakry and M.~Emery applied to a solution $u$ of~\eqref{HeatFlow}, the admissible range for $p$ is therefore~\eqref{Range:p} as shown in~\cite{MR808640,1504}: this is the range in which we know that $\gamma\ge0$. Since $\lim_{t\to+\infty}\mathsf e(t)=\lim_{t\to+\infty}\mathsf i(t)=0$ and $\frac d{dt}(\mathsf i-\,d\,\mathsf e)=\mathsf i'+2\,\,d\,\mathsf i\le0$, it is straightforward to deduce that $\mathsf i-\,d\,\mathsf e\ge0$ for any $t\ge0$ and, as a special case, at $t=0$ for an arbitrary initial datum. This completes the proof of~\eqref{Ineq:GNS}, after replacing $u$ by $|u|$ and removing the assumption $u>0$ by a density argument.

Following an idea of~\cite{MR2152502}, it has been observed in~\cite{DEKL} that an improvement is achieved for any $p\in[1,2)\cup(2,2^\#)$ using
\[
\mathsf i^2=\(\kern4pt\iSd{u\cdot\frac{|\nabla u|^2}u}\)^2\le\iSd{u^2}\iSd{\frac{|\nabla u|^4}{u^2}}=\(1-(p-2)\,\mathsf e\)\iSd{\frac{|\nabla u|^4}{u^2}}
\]
where the last equality holds if we impose that $\nrmSd up=1$ at $t=0$. This completes the proof of Lemma~\ref{Lem:ODE}.\end{proof}

\begin{lemma}\label{Lem:ODEsolution} For any $\gamma\ge0$, the solution $\varphi$ of
\be{ode}
\varphi'(s)=1+\frac{\gamma\,\varphi(s)}{1-(p-2)\,s}\,,\quad\varphi(0)=0\,,
\ee
is given by~\eqref{phifunction}.\end{lemma}
\begin{proof} The solution of~\eqref{ode} is unique and it is a straightforward computation that $\varphi$ given by~\eqref{phifunction} solves~\eqref{ode}.\end{proof}
\begin{lemma}\label{Lem:ODEvarphi} Let $d\ge1$, $\gamma$ be given by~\eqref{gamma1} and assume that $p$ is in the range~\eqref{Range:p}. Then~\eqref{improved} holds with $\varphi$ given by~\eqref{phifunction}.\end{lemma}
\begin{proof} With the notation of Lemma~\ref{Lem:ODE}, we compute
\[
2\,\frac d{dt}\(\mathsf i-\,d\,\varphi(\mathsf e)\)=-\(\mathsf e''+2\,d\,\mathsf e'\)-\,2\,d\,\mathsf e'\,\frac{\gamma\,\varphi(\mathsf e)}{1-(p-2)\,\mathsf e}\le-\,\frac{4\,\gamma\,\mathsf i}{1-(p-2)\,\mathsf e}\(\mathsf i-\,d\,\varphi(\mathsf e)\)
\]
using~\eqref{ode} in the equality and then~\eqref{BEode} in the inequality. Since $\lim_{t\to+\infty}\mathsf e(t)=\lim_{t\to+\infty}\mathsf i(t)=0$ and $\mathsf i-\,d\,\varphi(\mathsf e)\sim\mathsf i-\,d\,\mathsf e\ge0$ in the asymptotic regime as $t\to+\infty$, this proves that for functions $u$ satisfying $\nrmSd up=1$,
\[
\mathsf i\ge\,d\,\varphi(\mathsf e)\,.
\]
By homogeneity, this proves~\eqref{improved} for an arbitrary function $u$.\end{proof}

Theorem~\ref{Thm:Interp} is then obtained by replacing $\varphi$ in~\eqref{improved} by the expression in~\eqref{phifunction}. As noted in Section~\ref{Sec:Main}, Theorem~\ref{Thm:Stability} is a simple consequence of Theorem~\ref{Thm:Interp} and of the stereographic projection using the computations of Appendix~\ref{App:Stereographic}. Theorem~\ref{Thm:AFST} is also a straightforward consequence of~\cite[Theorem~2.2 and Proposition~5.4]{1504} using the stereographic projection. Hence all results of Section~\ref{Sec:Main} are established except Theorem~\ref{Thm:Lambda-Mu}.

A sharper version of Theorem~\ref{Thm:Stability}, valid for any $p$ in the range~\eqref{Range:p}, can be deduced directly from~\eqref{improved} with $\varphi$ given by~\eqref{phifunction} using the stereographic projection. It goes as follows.
\begin{proposition}\label{Prop:Stability} Let $d\ge2$ and assume that $p$ is in the range~\eqref{Range:p}. Then for any $v\in\mathrm L^2\big(\R^d,\cb x^{-4}\,dx\big)$ such that $\nabla v\in\mathrm L^2(\R^d,dx)$ we have
\begin{multline*}
\iRd{|\nabla v|^2}-d\,(d-2)\iRd{\frac{|v|^2}{\cb x^4}}\\
\ge\frac{4\,d}{2-p-\gamma}\left[\,\iRd{\frac{|v|^2}{\cb x^4}}-\kappa_p^{1-\frac\gamma{2-p}}\(\iRd{\frac{|v|^p}{\cb x^{\delta(p)}}}\)^{\frac2p\(1-\frac\gamma{2-p}\)}\(\iRd{\frac{|v|^2}{\cb x^4}}\)^{\frac\gamma{2-p}}\right]
\end{multline*}
if $\gamma\neq2-p$, and
\[
\iRd{|\nabla v|^2}-d\,(d-2)\iRd{\frac{|v|^2}{\cb x^4}}\ge\frac{8\,d}{p-2}\,\(\iRd{\frac{|v|^2}{\cb x^4}}\)\,\log\(\kappa_p^{-1}\,\frac{\iRd{\frac{|v|^2}{\cb x^4}}}{\iRd{\frac{|v|^p}{\cb x^{\delta(p)}}}}\)
\]
if $\gamma=2-p$, where $\kappa_p=2^{\frac{\delta(p)}p-2}\,\big|\S^d\big|^{1-\frac2p}$.\end{proposition}

\begin{remark}\label{Rem:Bianchi-Egnell} Inequalities~\eqref{improvedineq}-\eqref{improvedineqLog} are key estimates in this paper. Because of the convexity of the function $\varphi$ defined by~\eqref{phifunction}, we know that~\eqref{improvedineq} and~\eqref{improvedineqLog} are stronger than~\eqref{Ineq:GNS} and~\eqref{Ineq:logSob}, even if all these inequalities are optimal.

The fact that
\[
\frac1{2-p-\gamma}\(\nrmSd u2^2-\nrmSd up^{2-\frac{2\,\gamma}{2-p}}\,\nrmSd u2^{\frac{2\,\gamma}{2-p}}\)\ge\frac1{p-2}\(\nrmSd up^2-\nrmSd u2^2\)
\]
can be recovered using H\"older's inequality. For instance, if $p>2$, we know that $\nrmSd u2\le\nrmSd up$. By homogeneity, we can assume without loss of generality that $\nrmSd u2=1$ and $t=\nrmSd up^2\ge1$. With $\theta=\gamma/(p-2)$, this amounts to
\[
t^{1+\theta}-1\ge(1+\theta)\,(t-1)
\]
which is obviously satisfied for any $t\ge1$ because $\theta$ is nonnegative. Similar arguments apply if $p<2$, $p\neq p_*(d)$ and the case $p=p_*(d)$ is obtained as a limit case. The difference of the two sides in the inequality is the measure of the distance to the constants.

As in~\cite{MR1124290}, the stability can also be obtained in the stronger semi-norm $u\mapsto\iSd{|\nabla u|^2}$. We can indeed rewrite the improved inequality as
\[\label{firsttt}
\mathsf e\le\,\varphi^{-1}\(\frac{\mathsf i}d\)\,,
\]
for any $u$ satisfying $\nrmSd up^2=1$, and obtain that
\[\label{seconddd}
\mathsf i\ge\,d\,\mathsf e+\widetilde\psi(\mathsf i)\quad\mbox{where}\quad\widetilde\psi(\mathsf i)=\mathsf i-d\,\varphi^{-1}\(\frac{\mathsf i}d\)\ge0\,.
\]
\end{remark}

An explicit lower bound for $\mu(\lambda)$ has been obtained in~\cite[Proposition~8]{DolEsLa-APDE2014}. Let us recall it with a sketch of the proof, for completeness.
\begin{proposition}[\cite{DolEsLa-APDE2014}]\label{PropDolEsLa-APDE2014} Assume that $d\ge3$ and let $\theta=d\,\frac{p-2}{2\,p}$. Then
\[
\mu(\lambda)\ge\frac {p-2}d\(\frac14\,d\,(d-2)\)^\theta\(\frac{\lambda\,d}{p-2}\)^{1-\theta}\quad\forall\,\lambda\,\ge1\,.
\]\end{proposition}
Notice that this bound is limited to the case $d\ge3$ and $p\in(2,2^*)$.
\begin{proof} From H\"older's inequality $\nrmSd up\le\nrmSd u{2^*}^\theta\,\nrmSd u2^{1-\theta}$, we get that
\begin{multline*}
\frac{\nrmSd{\nabla u}2^2+\frac{\lambda\,d}{p-2}\,\nrmSd u2^2}{\nrmSd up^2}\\
\ge\(\frac{\nrmSd{\nabla u}2^2+\frac{\lambda\,d}{p-2}\,\nrmSd u2^2}{\nrmSd u{2^*}^2}\)^\theta\(\frac{\nrmSd{\nabla u}2^2+\frac{\lambda\,d}{p-2}\,\nrmSd u2^2}{\nrmSd u2^2}\)^{1-\theta}\,.
\end{multline*}
After dropping $\nrmSd{\nabla u}2^2$ in the second parenthesis of the right-hand side and observing that $1/(p-2)\ge(d-2)/4$, the conclusion holds using the Sobolev inequality in the first parenthesis. We indeed recall that $\mu(\lambda)=\frac14\,d\,(d-2)$ for any $\lambda\ge1$ if $p=2^*$.\end{proof}

We may notice that the estimate of Proposition~\ref{PropDolEsLa-APDE2014} captures the order in $\lambda$ of $\mu(\lambda)$ as $\lambda\to+\infty$ but is not accurate close to $\lambda=1$ and limited to the case $p\in(2,2^*)$ and $d\ge3$. It turns out that the whole range~\eqref{Range:p} for any $d\ge1$ can be covered as a consequence of Theorem~\ref{Thm:Interp} with a lower bound for $\mu(\lambda)$ which is increasing with respect to $\lambda\ge1$ and such that it takes the value $1$ if $\lambda=1$. This is essentially the contents of Theorem~\ref{Thm:Lambda-Mu} for $p\in(2,2^\#)$, which also covers the range $p\in[1,2)$.

\begin{proof}[Proof of Theorem~\ref{Thm:Lambda-Mu}] We shall distinguish several cases.

\smallskip\noindent\underline{1) Case $p\in(2,2^\#)$.}
Assume that $\lambda>1$ and $\theta>0$. We deduce from
\[
\underline\mu(\lambda):=\(\frac{(\theta+1)\,\lambda-1}\theta\)^\frac\theta{\theta+1}=\min_{t\ge1}\frac 1t\(\lambda+\frac{t^{1+\theta}-1}{1+\theta}\)
\]
that
\[
\frac{t^{1+\theta}-1}{1+\theta}\ge\underline\mu(\lambda)\,t-\lambda\quad\forall\,t\ge1\,.
\]
With $\theta=\frac\gamma{p-2}$, Inequality~\eqref{improvedineq} takes the form
\[
\frac{p-2}d\,\nrmSd{\nabla u}2^2\ge\frac1{1+\theta}\(\nrmSd up^{2\,(1+\theta)}\,\nrmSd u2^{-2\,\theta}-\nrmSd u2^2\)\quad\forall\,u\in\mathrm H^1(\S^d)\,.
\]
Using $t=\nrmSd up^2/\nrmSd u2^2\ge1$, the right-hand side satisfies
\begin{multline*}
\frac1{1+\theta}\(\nrmSd up^{2\,(1+\theta)}\,\nrmSd u2^{-2\,\theta}-\nrmSd u2^2\)=\frac{\nrmSd up^2}{1+\theta}\,\frac1t\(t^{1+\theta}-1\)\\
\ge\nrmSd up^2\(\underline\mu(\lambda)-\frac\lambda{t}\)=\underline\mu(\lambda)\,\nrmSd up^2-\lambda\,\nrmSd u2^2\,.
\end{multline*}
Hence we find
\[
\mu(\lambda)\ge\underline\mu(\lambda)=\(\lambda+\frac{p-2}\gamma\,(\lambda-1)\)^\frac\gamma{\gamma+p-2}\quad\forall\,\lambda\ge1\,.
\]

\smallskip\noindent\underline{2) Case $p\in\big(p_*(d),2\big)$.}
In this regime we have $\gamma>2-p$ and take $\theta=\frac\gamma{2-p}-1>0$. We deduce from
\[
\underline\lambda(\mu):=\frac{(\theta+1)\,\mu^{\frac\theta{\theta+1}}-1}{\theta}=\min_{t\in[0,1]}\(\frac{t^{-\theta}-1}{\theta}+\mu\,t\)
\]
that
\[
\frac{t^{-\theta}-1}{\theta}\ge\underline\lambda(\mu)-\mu\,t\quad\forall\,t\in[0,1]\,.
\]
Inequality~\eqref{improvedineq} takes the form
\[
\frac{2-p}d\,\nrmSd{\nabla u}2^2\ge\frac1\theta\(\nrmSd up^{-2\,\theta}\,\nrmSd u2^{2\,(1+\theta)}-\nrmSd u2^2\)\quad\forall\,u\in\mathrm H^1(\S^d)\,.
\]
Using $t=\nrmSd up^2/\nrmSd u2^2\le1$, the right-hand side satisfies
\begin{multline*}
\frac1\theta\(\nrmSd up^{-2\,\theta}\,\nrmSd u2^{2\,(1+\theta)}-\nrmSd u2^2\)=\nrmSd u2^2\,\frac{t^{-\theta}-1}\theta\\
\ge\nrmSd u2^2\(\underline\lambda(\mu)-\mu\,t\)=\underline\lambda(\mu)\,\nrmSd u2^2-\mu\,\nrmSd up^2\,.
\end{multline*}
Hence we find
\[
\lambda(\mu)\ge\underline\lambda(\mu)=\frac{2-p-\gamma\,\mu^{1-\frac{2-p}\gamma}}{2-p-\gamma}\quad\forall\,\mu\ge1\,.
\]

\smallskip\noindent\underline{3) Case $p=p_*(d)$.} It is achieved by taking the limit as $p\to p_*(d)$, but the estimate degenerates into $\lambda(\mu)\ge1$, which we already know because $\lambda(\mu)\ge\lambda(1)=1$ for any $\lambda\ge1$.

\smallskip\noindent\underline{4) Case $p\in\big(1,p_*(d)\big)$ and $d\neq2$.}
In this regime we have $\gamma<2-p$ and take $\theta=\frac\gamma{2-p}\in(0,1)$. We deduce from
\[
\underline\lambda(\mu):=\frac{1-\theta\,\mu^{1-\frac1\theta}}{1-\theta}=\min_{t\in[0,1]}\(\frac{1-t^{1-\theta}}{1-\theta}+\mu\,t\)
\]
that
\[
\frac{1-t^{1-\theta}}{1-\theta}\ge\underline\lambda(\mu)-\mu\,t\quad\forall\,t\in[0,1]\,.
\]
Inequality~\eqref{improvedineq} takes the form
\[
\frac{2-p}d\,\nrmSd{\nabla u}2^2\ge\frac1{1-\theta}\(\nrmSd u2^2-\nrmSd up^{2\,(1-\theta)}\,\nrmSd u2^{2\,\theta}\)\quad\forall\,u\in\mathrm H^1(\S^d)\,.
\]
Using $t=\nrmSd up^2/\nrmSd u2^2\le1$, the right-hand side satisfies
\begin{multline*}
\frac1{1-\theta}\(\nrmSd u2^2-\nrmSd up^{2\,(1-\theta)}\,\nrmSd u2^{2\,\theta}\)=\nrmSd u2^2\,\frac{1-t^{1-\theta}}{1-\theta}\\
\ge\nrmSd u2^2\(\underline\lambda(\mu)-\mu\,t\)=\underline\lambda(\mu)\,\nrmSd u2^2-\mu\,\nrmSd up^2\,.
\end{multline*}
Hence we find
\[
\lambda(\mu)\ge\underline\lambda(\mu)=\frac{2-p-\gamma\,\mu^{1-\frac{2-p}\gamma}}{2-p-\gamma}\quad\forall\,\mu\ge1\,.
\]
\end{proof}

\section{Inequalities based on nonlinear flows}\label{Sec:NonlinearFlows}

In this section, the range of $p$ is
\be{Range:p-NL}
p\in[1,2^*],\; p\ne 2 \quad\mbox{if}\quad d\ge3\quad\mbox{and}\quad p\in[1,+\infty), \; p\ne 2\quad\mbox{if}\quad d=1\,,2\,.
\ee
This range includes in particular the case \hbox{$2^\#< p<2^*$}, which was not covered in Section~\ref{Sec:Heat}. As in~\cite{MR2381156,DEKL,1504}, let us replace~\eqref{HeatFlow} by the nonlinear diffusion equation
\be{NLeqn}
\frac{\partial u}{\partial t}=u^{2-2\beta}\(\Delta u+\kappa\,\frac{|\nabla u|^2}u\)\,.
\ee
The parameter $\beta$ has to be chosen appropriately as we shall see below. With the choice $\kappa=\beta\,(p-2)+1$, one can check that
\[
\frac d{dt}\iSd{u(t,\cdot)^{\beta\,p}}=0
\]
because $\rho=u^{\beta\,p}$ solves the porous medium equation $\frac{\partial\rho}{\partial t}=\Delta\rho^m$ with $m$ such that
\be{Eqn:m-beta}
\frac1\beta+\frac p2=1+m\,\frac p2\,.
\ee
Notice that $m>0$ can be larger or smaller than $1$ depending on $\beta$, $d$ and $p$. The \emph{entropy} and the \emph{Fisher information} are redefined respectively by
\[
\mathsf e:=\frac 1{p-2}\(\bignrmSd{u^\beta}p^2-\bignrmSd{u^\beta}2^2\)\quad\mbox{and}\quad\mathsf i:=\bignrmSd{{\nabla u^\beta}}2^2\,.
\]
The equation $\mathsf e'=-\,2\,\mathsf i$ holds true only if $\beta=1$, in which case~\eqref{NLeqn} coincides with~\eqref{HeatFlow}. Here we have: $\mathsf e'=-\,2\,\beta^2\,\nrmSd{\nabla u}2^2\neq-\,2\,\mathsf i$ if $\beta\neq1$ but we can still compute $\frac d{dt}(\mathsf i-\,d\,\mathsf e)$ and obtain that
\be{ineqqq}
\frac1{2\,\beta^2}\(\mathsf i'-\,d\,\mathsf e'\)=-\,\frac d{d-1}\iSd{\left\|\mathrm L u-\beta\,(p-1)\,\frac{d-1}{d+2}\,\mathrm M u\right\|^2}-\gamma(\beta)\iSd{\frac{|\nabla u|^4}{u^2}}
\ee
with
\be{gamma}
\gamma(\beta):=-\(\frac{d-1}{d+2}\,(\kappa+\beta-1)\)^2+\,\kappa\,(\beta-1)+\,\frac d{d+2}\,(\kappa+\beta-1)\,.
\ee
To guarantee that $\gamma(\beta)\ge0$ for some $\beta\in\R$, a discussion has to be made: see Lemma~\ref{Lem:RangeBeta} below for a detailed statement and also~\cite{DEKL}. Notice that the value of $\gamma$ given by~\eqref{gamma1} in Sections~\ref{Sec:Main} and~\ref{Sec:Heat} corresponds to~\eqref{gamma} with $\beta=1$. In the sequel let us denote by $\mathfrak B(p,d)$ the set of $\beta$ such that $\gamma(\beta)\ge0$ with $p$ in the range~\eqref{Range:p-NL}.
\begin{lemma} Let $d\ge1$ and assume that $p$ is in the range~\eqref{Range:p-NL}. Then $\mathfrak B(p,d)$ is non-empty.\end{lemma}
\begin{proof} As a function of $\beta$, $\gamma(\beta)$ is a polynomial of degree at most two. We refer to~\cite[Appendix~A]{DEKL} for a proof, up to the restriction $p<9+4\,\sqrt3$ in dimension $d=2$. If $d=2$ and $p>9+4\,\sqrt3$, we can make the choice $\beta=4\,(5-p)/(p^2-18\,p+33)$ which corresponds to $m=8\,(p-1)/(p^2-18\,p+33)$, while for $d=2$ and $p=9+4\,\sqrt3$, $\beta\ge-1/(2+2\,\sqrt3)$ is an admissible choice (in that case, $\gamma(\beta)$ is a polynomial of degree~$1$).\end{proof}

\begin{corollary}\label{Cor:BE-NL} Let $d\ge1$ and assume that $p$ is in the range~\eqref{Range:p-NL}. For any $\beta\in\mathfrak B(p,d)$, any solution of~\eqref{NLeqn} is such that $\mathsf i-\,d\,\mathsf e$ is monotone non-increasing with limit $0$ as $t\to+\infty$.
\end{corollary}
As a consequence, we know that $\mathsf i\ge d\,\mathsf e$, which proves~\eqref{Ineq:GNS} in the range~\eqref{Range:p-NL}. Let us define by
\[
\beta_\pm(p,d):=\frac{d^2-d\,(p-5)-\,2\,p+6\pm(d+2)\,\sqrt{d\,(p-1)\,\big(2\,d-p\,(d-2)\big)}}{d^2\(p^2-3\,p+3\)-\,2\,d\,(p^2-3)+(p-3)^2}
\]
the roots of $\gamma(\beta)=0$, provided $d^2\(p^2-3\,p+3\)-\,2\,d\,(p^2-3)+(p-3)^2\neq0$, \emph{i.e.},
\[\label{Range:p-NL-excl}
\begin{array}{ll}
p\neq 9\pm4\,\sqrt3\quad&\mbox{if}\quad d=2\,,\\
p\neq\frac94\quad\mbox{and}\quad p\neq6\quad&\mbox{if}\quad d=3\,,\\
p\neq3\quad&\mbox{if}\quad d=4\,.
\end{array}
\]
The precise description of $\mathfrak B(p,d)$ goes as follows.
\begin{lemma}\label{Lem:RangeBeta} Let $d\ge1$ and assume that $p$ is in the range~\eqref{Range:p-NL}. The set $\mathfrak B(p,d)$ with $p$ is defined by
\begin{enumerate}
\item[(i)] if $d=1$, $\beta_-(p,1)\le\beta\le\beta_+(p,1)$ if $p<2$, $\beta\le3/4$ if $p=2$ and $\beta\in(-\infty,\beta_+(p,1)\big]\cup\big[\beta_-(p,1),+\infty)$ if $p>2$.
\item[(ii)] if $d=2$, $\beta_-(p,1)\le\beta\le\beta_+(p,1)$ if $p<9-4\,\sqrt3$ or $p>9+4\,\sqrt3$, $\beta\le
1/(2\,\sqrt3-2)$ if $p=9-4\,\sqrt3$, $\beta\in(-\infty,\beta_+(p,1)\big]\cup\big[\beta_-(p,1),+\infty)$ if $9-4\,\sqrt3<p<9+4\,\sqrt3$ and $\beta\ge-1/(2\,\sqrt3+2)$ if $p=9+4\,\sqrt3$.
\item[(iii)] if $d=3$, $\beta_-(p,1)\le\beta\le\beta_+(p,1)$ if $p<9/4$, $\beta\in(-\infty,\beta_+(p,1)\big]\cup\big[\beta_-(p,1),+\infty)$ if $9/4<p<6$ and $\beta\le2/3$ if $p=9/4$.
\item[(iv)] if $d\ge4$, $\beta_-(p,d)\le\beta\le\beta_+(p,d)$ if $(d,p)\neq(4,3)$ and $\beta\ge\beta_-(p,d)$ if $(d,p)=(4,3)$.
\end{enumerate}
\end{lemma}
A much simpler picture is obtained in terms of $m=m(\beta,p,d)$ given by~\eqref{Eqn:m-beta}. Let $m_-(p,d)=\min_\pm m\big(\beta_\pm(p,d),p,d\big)$ and $m_+(p,d)=\max_\pm m\big(\beta_\pm(p,d),p,d\big)$. The completion of the set $\big\{m(\beta,p,d)\,:\,\beta\in\mathfrak B(p,d)\big\}$ is simply the set
\[
m_-(p,d)\le m\le m_+(p,d)\,.
\]
See Fig.~\ref{F2}.

\medskip As observed in~\cite{MR2381156,DEKL,1504}, an improved inequality can also be obtained. Since the case $p\in[1,2)$ is covered in Section~\ref{Sec:Heat}, we shall assume from now on that $p>2$. With
\[\label{varphi_beta}
\varphi_\beta(s)=\int_0^s\exp\(\frac{2\,\gamma(\beta)}{\beta\,(\beta-1)\,p}\(\(1-(p-2)\, s\)^{1-\zeta-\frac1{2\beta}}-\(1-(p-2)\,z\)^{1-\zeta-\frac1{2\beta}}\)\)\,dz\,,
\]
where $\gamma=\gamma(\beta)$ is given by~\eqref{gamma} and $\zeta=\zeta(\beta)=\frac{2-\,(4-p)\,\beta}{2\,\beta\,(p-2)}$, let us consider
\be{optimalphi}
\varphi(s):=\sup\Big\{\varphi_\beta(s)\,:\,\beta\in\mathfrak B(p,d)\Big\}\,.
\ee
\begin{theorem}\label{Thm:varphi} Let $d\ge1$ and assume that $p\in(2,2^*)$. Inequality~\eqref{improved} holds with $\varphi$ defined by~\eqref{optimalphi}.\end{theorem}
\begin{proof} Using the identity $\frac12+\frac{\beta-1}{\beta\,(p-2)}+\zeta=1$, H\"older's inequality shows that
\begin{multline*}
\frac1{\beta^2}\iSd{\big|\nabla\big(u^\beta\big)\big|^2}=\iSd{u^{2(\beta-1)}\,|\nabla u|^2}=\iSd{\frac{|\nabla u|^2}u\cdot u^\frac{p(\beta-1)}{p-2}\cdot u^{2\beta\zeta}}\\
\le\(\kern4pt\iSd{\frac{|\nabla u|^4}{u^2}}\)^\frac12\(\kern4pt\iSd{u^{\beta p}}\)^\frac{\beta-1}{\beta\,(p-2)}\(\kern4pt\iSd{u^{2\beta}}\)^\zeta\,.
\end{multline*}
With the choice $\bignrmSd {u^\beta}p=1$,
we find that
\[
\(\kern4pt\iSd{\frac{|\nabla u|^4}{u^2}}\)^{1/2}\ge\frac{1}{\beta^2}\,\frac{i}{\(1-(p-2)\,\mathsf e\)^{\zeta}}\,.
\]
On the other hand, by using the identity $\frac12+\frac{\beta-1}{2\,\beta}+\frac1{2\,\beta}=1$,
and H\"older's inequality again, we have also
\[
\(\kern4pt\iSd{\frac{|\nabla u|^4}{u^2}}\)^{1/2}\ge\frac{\iSd{|\nabla u|^2}}{\(1-(p-2)\,\mathsf e\)^{\frac1{2\,\beta}}}\,,
\]
since $d\mu$ is a probability measure on $\S^d$. Therefore, from~\eqref{ineqqq} we get the inequality
\[\label{ineqbetaa}
\frac d{dt}\(\mathsf i-d\,\mathsf e\)\le\frac{\gamma(\beta)\,\mathsf i\,\mathsf e'}{\beta^2\(1-(p-2)\,\mathsf e\)^{\zeta+\frac1{2\,\beta}}}\,.
\]
For every $\beta>1$ it is possible to find a function $\psi_\beta$ satisfying the ODE
\[
\frac{\psi_\beta''(s)}{\psi_\beta'(s)}=-\,\frac{\gamma(\beta)}{\beta^2}\(1\,-\,(p-2)\,s\)^{-\zeta-\frac1{2\beta}}\,,\quad\psi_\beta(0)=0\,,
\]
with $\zeta=\zeta(\beta)$, such that $\psi_\beta'>0$. Then
\[\label{ineqqbeta}
\frac d{dt}\big(\mathsf i\,\psi_\beta'(\mathsf e)-d\,\psi_\beta(\mathsf e)\big)\le0\,,
\]
from which we conclude that $\mathsf i\ge d\,\varphi_\beta(\mathsf e)$ with $\varphi_\beta:=\psi_\beta/\psi_\beta'$. It is then elementary to check that $\varphi_\beta$ satisfies the ODE
\[
\varphi_\beta'=1-\varphi_\beta\,\frac{\psi_\beta''(s)}{\psi_\beta'(s)}=1+\frac{\gamma(\beta)}{\beta^2}\(1\,-\,(p-2)\,s\)^{-\zeta-\frac1{2\beta}}\,\varphi_\beta
\]
and that $\varphi_\beta(0)=0$. Solving this linear ODE, we find the expression of $\varphi_\beta$. Notice that $\varphi_\beta$ is defined for any $s\in\big[0,1/(p-2)\big)$ and that $\varphi_\beta(s)>0$ for any $s\neq0$. From the equation satisfied by $\varphi_\beta$ we get that $\varphi'_\beta(s)>1$ and $\varphi''_\beta(s)>0$, hence $\varphi_\beta(s)>s$ for any admissible $\beta$ and any $s\in\big(0,1/(p-2)\big)$.\end{proof}

Let us define
\be{overline_mu}
\underline\mu(\lambda)=\min_{t\ge1}\left[\frac{p-2}t\,\varphi\(\frac{t-1}{p-2}\)+\frac\lambda t\right]\,.
\ee
By arguing exactly as in the proof of Theorem~\ref{Thm:Lambda-Mu}, we obtain an estimate of the optimal constant in~\eqref{Ineq:GNSgen1} which is valid for instance if $2^\#<p<2^*$.
\begin{corollary}\label{Cor:Lambda-Mu} Let $d\ge1$ and assume that $p\in(2,2^*)$. With the notations of Theorem~\ref{Thm:varphi} and $\underline\mu(\lambda)$ defined by~\eqref{overline_mu}, the optimal constant in~\eqref{Ineq:GNSgen1} can be estimated for any $p\in(2,2^*)$ by
\[
\mu(\lambda)\ge\underline\mu(\lambda)\quad\forall\,\lambda\ge1\,.
\]
\end{corollary}

Another consequence is that one can write an improved inequality on $\R^d$ in the spirit of Proposition~\ref{Prop:Stability}, for any $p\in(1,2^*)$, $p\ne2$. Since the expression involves $\varphi$ as defined in Theorem~\ref{Thm:varphi}, we do not get any fully explicit expression, so we shall leave it to the interested reader. A major drawback of our method is that $\varphi$ is defined through a primitive. With some additional work, $\varphi$ can be written as an incomplete $\Gamma$ function, which is however not of much practical interest. This is why it is interesting to consider a special case, for which we obtain an explicit control of the remainder term. For completeness, let us state the following result which applies to a particular class of functions $u$.
\begin{theorem}[\cite{1504}]\label{Thm:Antipodal} Let $d\ge3$. If $p\in(1,2)\cup(2,2^*)$, we have
\[
\iSd{|\nabla u|^2}\ge\frac d{p-2}\left[1+\frac{(d^2-4)\,(2^*-p)}{d\,(d+2)+p-1}\right]\(\nrmSd up^2-\nrmSd u2^2\)
\]
for any $u\in\mathrm H^1(\S^d,d\mu)$ with \emph{antipodal symmetry}, i.e.,
\be{Eqn:antipodal}
u(-x)=u(x)\quad\forall\,x\in\S^d\,.
\ee
The limit case $p=2$ corresponds to the improved logarithmic Sobolev inequality
\[
\iSd{|\nabla u|^2}\ge\frac d2\frac{(d+3)^2}{(d+1)^2}\iSd{|u|^2\,\log\(\frac{|u|^2}{\nrmSd u2^2}\)}
\]
for any $u\in\mathrm H^1(\S^d,d\mu)\setminus\{0\}$ such that~\eqref{Eqn:antipodal} holds.\end{theorem}
We refer to \cite[Theorem~5.6]{1504} and its proof for details. Instead of~\eqref{Eqn:antipodal}, one can use any symmetry which guarantees that $\frac d{dt}\iSd{u(t,\cdot)^{\beta\,p}}=0$ if we evolve $u$ according to~\eqref{NLeqn}. Using the stereographic projection, one can obtain a weighted inequality with the same constant on $\R^d$, for solutions which have the \emph{inversion symmetry} corresponding to~\eqref{Eqn:antipodal}.

\section{Further results and concluding remarks}

The interpolation inequalities \eqref{Ineq:GNSgen1} and~\eqref{Ineq:GNSgen2} are equivalent to \emph{Keller-Lieb-Thirring} estimates for the principal eigenvalue of Schr\"odinger operators, respectively $-\Delta-V$ on $\S^d$ with $V\ge0$ in $\mathrm L^q(\S^d)$ for some $q>1$, and $-\Delta+V$ on $\S^d$ with $V>0$ such that $V^{-1}\in\mathrm L^q(\S^d)$, again for some $q>1$. See for instance~\cite{Dolbeault2013437,DolEsLa-APDE2014} and references therein.
\begin{corollary}\label{Cor:KLT1} Let $d\ge1$, $q>\max\{1,d/2\}$, $p=2\,q/(q-1)$ and assume that $V$ be a positive potential in $\mathrm L^q(\S^d)$ with $\mu=\nrmSd Vq$. If $\underline\lambda(\mu)$ denotes the inverse of $\lambda\mapsto\underline\mu(\lambda)$ defined by~\eqref{overline_mu} for some convex function $\varphi$ such that~\eqref{improved} holds with $\varphi(0)=0$ and $\varphi'(0)=1$, then
\[
\lambda_1(-\Delta-V)\ge-\,\underline\lambda\(\nrmSd Vq\)\,.
\]
\end{corollary}
\begin{proof} From H\"older's inequality $\iSd{V\,u^2}\le\mu\,\nrmSd up^2$ with $\mu=\nrmSd Vq$, we learn that
\[
\frac{\iSd{\(|\nabla u|^2-V\,u^2\)}}{\nrmSd u2^2}\ge\frac{\nrmSd{\nabla u}2^2-\mu\,\nrmSd up^2}{\nrmSd u2^2}
\ge-\,\underline\lambda(\mu)\,.
\]
\end{proof}
Corollary~\ref{Cor:KLT1} applies to $\varphi$ defined by~\eqref{optimalphi} for any $p\in(2,2^*)$ and to $\varphi$ defined by~\eqref{phifunction} for any $p\in(2,2^\#)$. In that case, the result holds with
\[
\underline\lambda(\mu)=\mu\quad\mbox{if}\quad\mu\in(0,1]\quad\mbox{and}\quad\underline\lambda(\mu)=\frac{p-2+\gamma\,\mu^{1+\frac{p-2}\gamma}}{p-2+\gamma}\quad\mbox{if}\quad\mu>1\,.
\]
Even more interesting is the fact that a result can also be deduced from Theorem~\ref{Thm:Lambda-Mu} in the range $p\in[1,2)$, $p\neq p_*(d)$, for which no explicit estimate was known so far. In that case, let us define
\[
\underline\lambda(\mu)=\mu\quad\mbox{if}\quad\mu\in(0,1]\quad\mbox{and}\quad\underline\lambda(\mu)=\frac{2-p-\gamma\,\mu^{1-\frac{2-p}\gamma}}{2-p-\gamma}\quad\mbox{if}\quad\mu>1\,.
\] \nc
\begin{corollary}\label{Cor:KLT2} Let $d\ge1$, $q>1$, $p=2\,q/(q+1)$ and assume that $V$ be a positive potential such that $V^{-1}\in \mathrm L^q(\S^d)$. Then
\[
\lambda_1(-\Delta-V)\ge \underline\lambda\(\nrmSd Vq\)\,.
\]
\end{corollary}
\begin{proof} By the reverse H\"older inequality, with $\mu=\bignrmSd{V^{-1}}q^{-1}$ we have
\[
\iSd{\(|\nabla u|^2+V\,|u|^2\)}\ge\nrmSd{\nabla u}2^2+\mu\,\nrmSd up^2\,.
\]
The conclusion holds using~\eqref{Ineq:GNSgen2} and Theorem~\ref{Thm:Lambda-Mu}, (i) .\end{proof}

\medskip Let us conclude with a summary and some considerations on open problems. This paper is devoted to improvements of~\eqref{Ineq:GNS} and~\eqref{Ineq:logSob} by taking into account additional terms in the \emph{carr\'e du champ} method. The stereographic projection then induces improved weighted inequalities on the Euclidean space $\R^d$. Alternatively, various improvements have been obtained on $\R^d$ using the scaling invariance: see for instance~\cite{MR3493423} and references therein. It is to be expected that these two approaches are not unrelated as well as nonlinear diffusion flows on $\S^d$ and nonlinear diffusion flows on $\R^d$ can probably be related. The self-similar changes of variables based on the so-called Barenblatt solutions also points in this direction: see~\cite{DEL-JEPE}. Concerning \emph{stability} issues, we have been able to establish various estimates with explicit constants, which are all limited to the subcritical range $p<2^*$ when $d\ge3$. This is clearly not optimal (see~\cite{MR1124290,1504}). A last point deserves to be mentioned: improved entropy - entropy production estimates like $\mathsf i\ge\,d\,\varphi(\mathsf e)$ mean increased convergence rates in evolution problems like~\eqref{HeatFlow} or~\eqref{NLeqn}: how to connect an initial time layer with large entropy $\mathsf e$ to an asymptotic time layer with an improved spectral gap obtained, for instance, by \emph{best matching} (which amounts to impose additional orthogonality conditions for large time asymptotics), is a topic of active research.

\begin{center}
\rule{2cm}{0.5pt}
\end{center}
\vspace*{-2cm}

\section*{\begin{center}Appendices\end{center}}
\appendix\section{Estimating the distance to the constants}\label{App:Distance}

In Section~\ref{Sec:Intro}, we claimed that the entropy
\[
u\mapsto\frac{\nrmSd up^2-\nrmSd u2^2}{p-2}
\]
is an estimate of the distance of the function $u$ to the constant functions. Let us give some details.

If $p\in[1,2)$ we know that
\[
\nrmSd u2^2-\nrmSd up^2\ge\frac{2-p}{2^{p-1}\,p^2}\,\nrmSd u2^{2\,(1-p)}\(\kern4pt\iSd{\left||u|^p-\overline u^p\right|^\frac2p}\)^p
\]
with $\overline u=\nrmSd up$, for any $u\in\mathrm L^p\cap\mathrm L^2(\S^d)$, by the \emph{generalized Csisz\'ar-Kullback-Pinsker inequality}: see \cite{MR1801751,Caceres-Carrillo-Dolbeault02} or \cite[Proposition~2.1]{doi:10.1142/S0218202518500574}, and references therein.

If $p>2$, let us define the constant
\[
c_q:=\inf_{t\in\R^+\setminus\{1\}}\frac{t^q-1-q\,(t-1)}{\nu_q(t-1)}\quad\mbox{with}\quad\nu_q(t)
=\left\{\begin{array}{ll}
|s|^2\quad&\mbox{if}\quad |s|\le1\\
|s|^q\quad&\mbox{if}\quad s>1
\end{array}\right.
\]
for any $q>1$. Let $q=p/2$ and use the above constant to get, with $t=u^2/\nrmSd u2^2$, the estimate
\[
\iSd{|u|^p}\ge\nrmSd u2^p\(1+c_{p/2}\iSd{\nu_{\!p/2}\(\frac{|u|^2}{\nrmSd u2^2}-1\)}\)
\]
and deduce that
\[
\nrmSd up^2-\nrmSd u2^2\ge\nrmSd u2^2\left[\(1+c_{p/2}\iSd{\nu_{\!p/2}\(\frac{|u|^2-\overline u^2}{\overline u^2}\)}\)^{2/p}-1\right]
\]
with $\overline u=\nrmSd u2$, for any $u\in\mathrm L^p\cap\mathrm L^2(\S^d)$. Although there is no good homogeneity property because of the definition of the function $\nu_{\!p/2}$, the right-hand side is clearly a measure of the distance of $u$ to the constant $\overline u$.

\section{Stereographic projection}\label{App:Stereographic}

Let $x\in\R^d$, $r=|x|$, $\omega=\frac x{|x|}$ and denote by $(\rho\,\omega,z)\in\R^d\times(-1,1)$ the cartesian coordinates on the unit sphere $\S^d\subset\R^{d+1}$ given by
\[
z=\frac{r^2-1}{r^2+1}=1-\frac 2{\cb x^2}\,,\quad\rho=\frac{2\,r}{\cb x^2}\,.
\]
Let $u$ be a function defined on $\S^d$ and consider its counterpart $v$ on $\R^d$ given by
\[
u(\rho\,\omega,z)=\(\frac{\cb x^2}2\)^\frac{d-2}2v(x)\quad\forall\,x\in\R^d\,.
\]
Recall that $\delta(p)=2\,d-p\,(d-2)$. For any $p\ge1$, we have
\[
\iSd{|u|^p}=\big|\S^d\big|^{-1}\,2^\frac{\delta(p)}2\iRd{\frac{|v|^p}{\cb x^{\delta(p)}}}
\]
and also
\[
\iSd{|\nabla u|^2}+\frac 14\,d\,(d-2)\iSd{|u|^2}=\big|\S^d\big|^{-1}\iRd{|\nabla v|^2}\,.
\]

\begin{acknowledgement}
This work has been partially supported by the Project EFI (J.D., ANR-17-CE40-0030) of the French National Research Agency (ANR).\\[-2pt]
{\sl\scriptsize\copyright~2019 by the authors. This paper may be reproduced, in its entirety, for non-commercial purposes.}
\end{acknowledgement}




\clearpage\section*{Figures}\label{Sec:Figures}

\begin{figure}[h]
\begin{center}\vspace*{-0.5cm}
\includegraphics{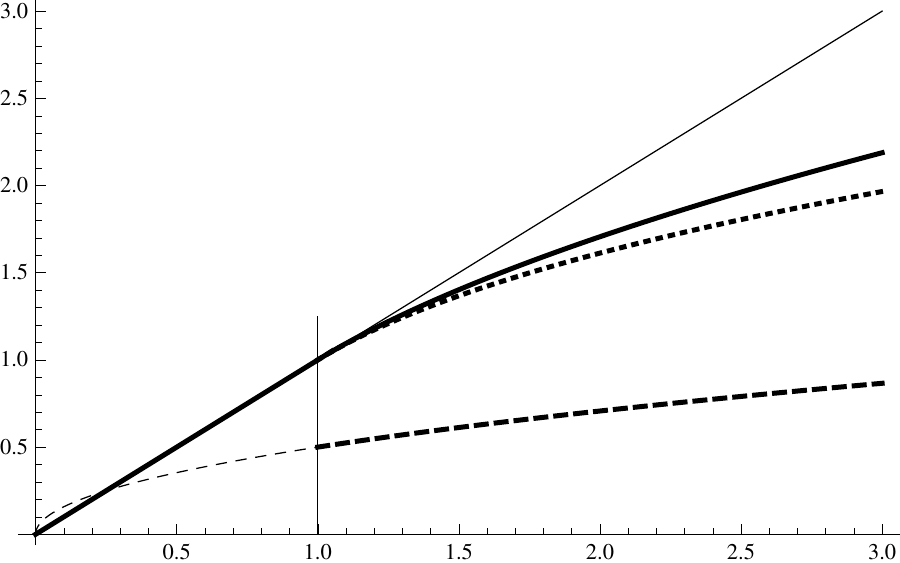}
\caption{\label{F1} The best constant $\lambda\mapsto\mu(\lambda)$ in Inequality~\eqref{Ineq:GNSgen1} for $d=3$ and $p=3$ is represented by the plain curve (numerical computation). The dashed line is the estimate of Proposition~\ref{PropDolEsLa-APDE2014} (valid only for $\lambda\ge1$) and the dotted line is the estimate of Theorem~\ref{Thm:Lambda-Mu}.}
\vspace*{1.5cm}
\includegraphics[width=3.5cm]{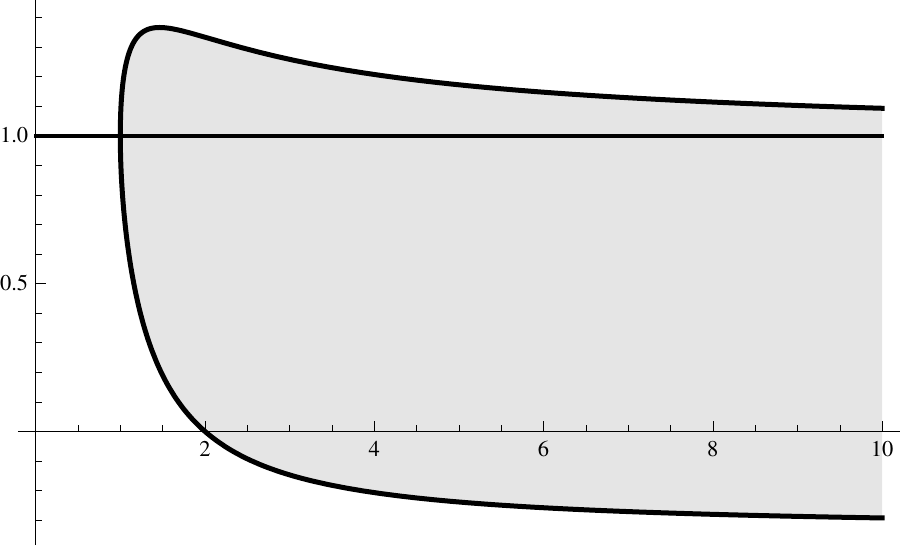}\hspace*{6pt}\includegraphics[width=3.5cm]{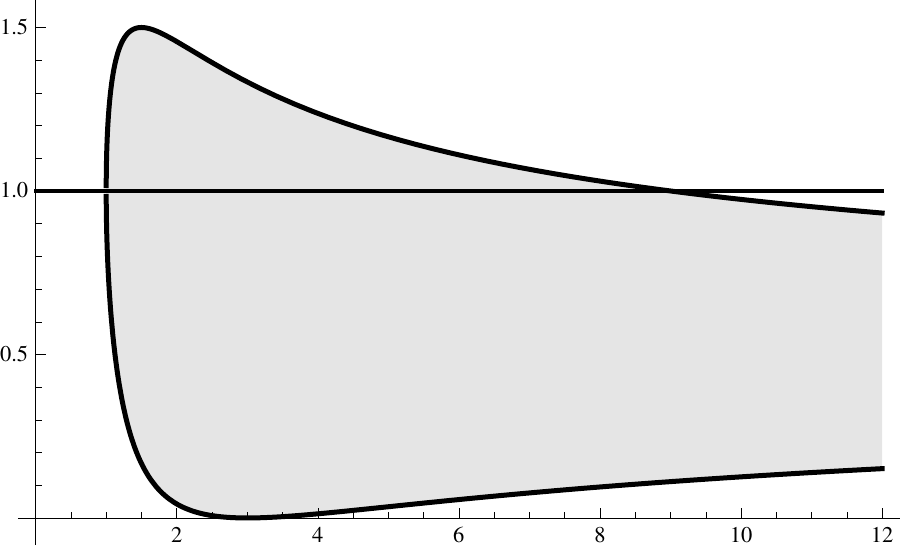}\hspace*{6pt}\includegraphics[width=3.5cm]{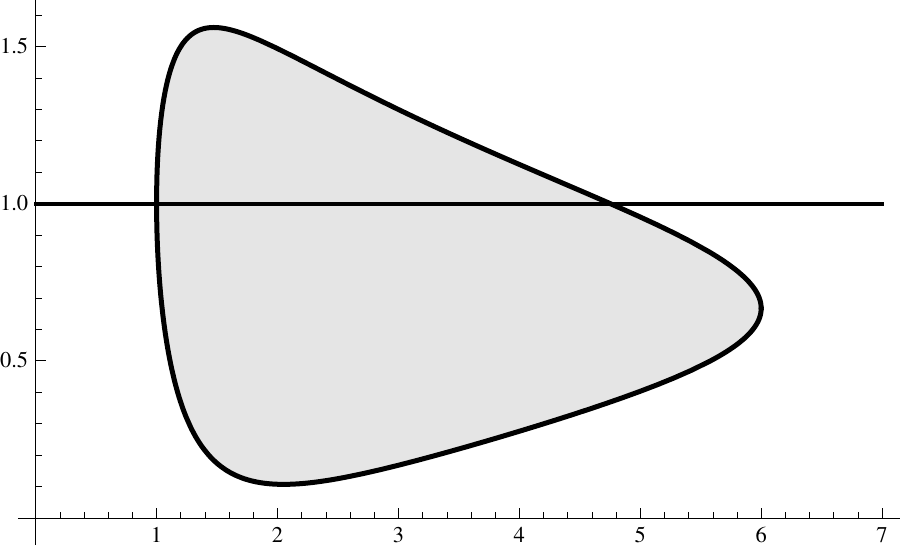}\\[8pt]\includegraphics[width=3.5cm]{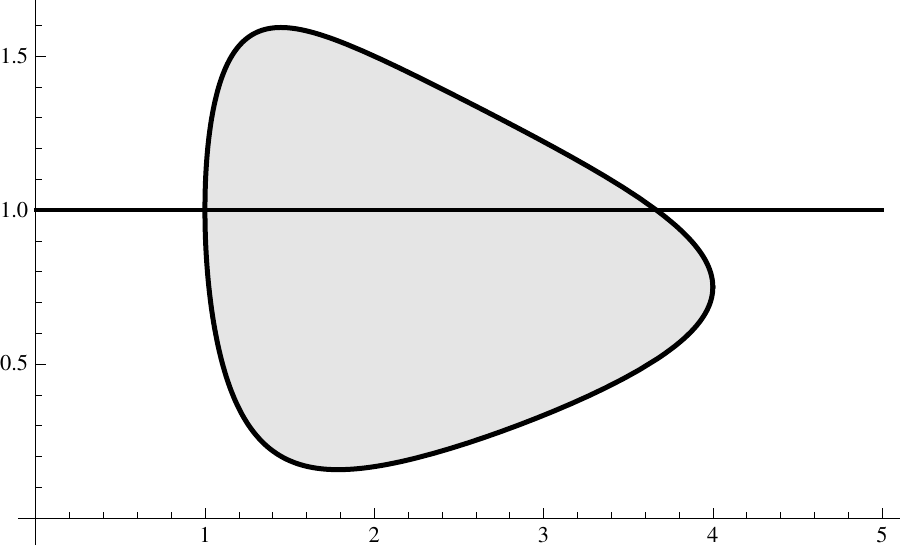}\hspace*{6pt}\includegraphics[width=3.5cm]{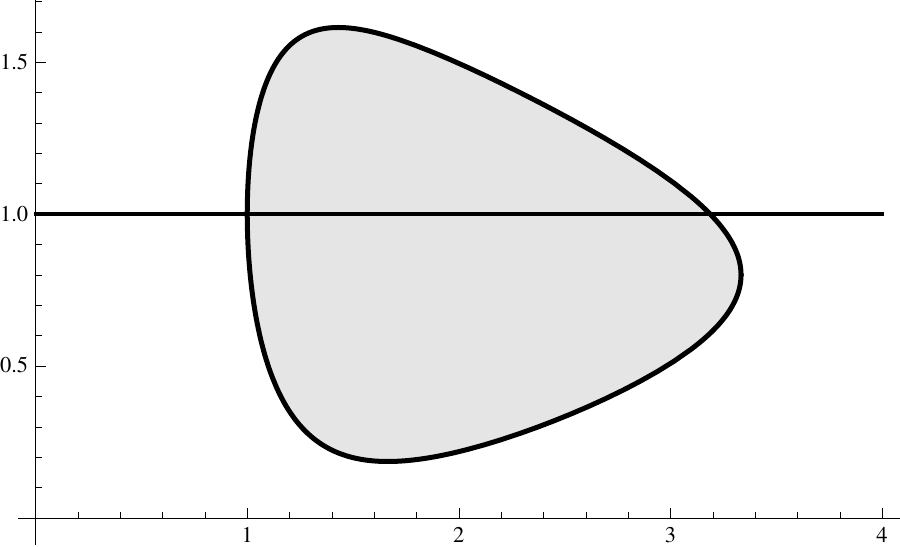}\hspace*{6pt}\includegraphics[width=3.5cm]{F2-5.pdf}
\caption{\label{F2} The admissible range for $d=1$, $2$, $3$ (first line), and $d=4$, $5$ and $10$ (from left to right), as it is deduced from Lemma~\ref{Lem:RangeBeta} using~\eqref{Eqn:m-beta}: the curves $p\mapsto m_\pm(p)$ enclose the admissible range of the exponent $m$.}
\end{center}\vspace*{-1cm}
\end{figure}

\end{document}